\documentclass[a4paper,12pt]{amsart}
\usepackage[latin1]{inputenc}
\usepackage{exscale}
\usepackage{amsmath}
\usepackage{amsfonts}
\usepackage{amssymb}
\usepackage{mathrsfs}
\usepackage{verbatim}
\usepackage{enumerate}
\usepackage{bbm}
\usepackage{graphics}
\usepackage{vmargin}

\usepackage{ifpdf}
\ifpdf 
\usepackage[pdftex]{hyperref} 
\else
\usepackage[hypertex]{hyperref}
\fi

\hypersetup{
citecolor=green,
menucolor=red,
citebordercolor=0 1 0,
linkbordercolor=1 0 0,
pdfpagemode=UseOutlines,
pdftitle={A note on Berestycki-Cazenave's classical instability result
for nonlinear Schr\"odinger equations},
pdfauthor={Stefan LE COZ <slecoz@univ-fcomte.fr>},
pdfsubject={A note on Berestycki-Cazenave's classical instability result
for nonlinear Schr\"odinger equations},
pdfkeywords={nonlinear Schrödinger equations; standing waves; instability; blow-up; variational methods}
}

\setmarginsrb {3.5cm}{3cm}{3.5cm}{4cm} {1cm}{1cm}{1cm}{1cm}

\newtheorem{defi}{Definition}
\newtheorem{thm}{Theorem}

\newtheorem{remq}[defi]{Remark}

\newtheorem{lem}[defi]{Lemma}

\newcommand{\nld}[1]{\|#1\|_2}
\newcommand{\nldd}[1]{\|#1\|_2^2}

\newcommand{\nhu}[1]{\|#1\|_{\hu}}

\renewcommand{\a}{\alpha}
\renewcommand{\b}{\beta}
\renewcommand{\d}{\delta}
\newcommand{\e}{\varepsilon}
\newcommand{\f}{\varphi}

\renewcommand{\k}{\chi}
\renewcommand{\l}{\lambda}

\newcommand{\w}{\omega}

\newcommand{\D}{\Delta}

\newcommand{\G}{\Gamma}

\newcommand{\R}{\mathbb{R}}
\newcommand{\RN}{\mathbb{R}^N}
\newcommand{\C}{\mathbb{C}}

\newcommand{\hu}{H^1(\RN)}
\newcommand{\hmu}{H^{-1}(\RN)}

\newcommand{\fl}{\f^{\l}}

\newcommand{\uo}{u_0}
\newcommand{\ul}{u^{\l}}

\newcommand{\vb}{v^{\b}}
\newcommand{\vbo}{v^{\bo}}
\newcommand{\vbot}{v^{\bot}}

\newcommand{\vl}{v^{\l}}
\newcommand{\vlo}{v^{\lo}}
\newcommand{\vlu}{v^{\lu}}

\newcommand{\dtt}{\partial_{tt}}
\newcommand{\intrn}{\int_{\RN}}

\newcommand{\lo}{\l_0}
\newcommand{\lu}{{\l_1}}
\newcommand{\lnd}{{\l^{\frac{N}{2}}}}
\newcommand{\bo}{\b_0}
\renewcommand{\bot}{\tilde{\bo}}

\newcommand{\dM}{d_{\M}}
\newcommand{\dw}{d(\w)}
\newcommand{\M}{\mathscr{M}}

\newcommand{\Tuo}{T_{\uo}}

\renewcommand{\leq}{\leqslant}
\renewcommand{\geq}{\geqslant}

\newcommand{\goesto}{\rightarrow}

\subjclass[2000]{35Q55,(35B35,35A15)}

\begin{document}

\title[On Berestycki-Cazenave's instability result
for NLS]{A note on Berestycki-Cazenave's classical instability result
for nonlinear Schr\"odinger equations}

\date\today

\author{Stefan LE COZ}
\address[Stefan LE COZ]{Laboratoire de Math\'{e}matiques,
\endgraf
Universit\'{e} de Franche-Comt\'{e}, 25030 Besan\c{c}on, FRANCE}
\email{slecoz@univ-fcomte.fr}

\begin{abstract}
In this note we give an alternative, shorter proof of the
classical result of Berestycki and Cazenave on the
instability by blow-up for the standing waves of some nonlinear
Schr\"odinger equations.
\end{abstract}

\maketitle

\section{Introduction}
In 1981, in a celebrated note \cite{bc1}, Berestycki and Cazenave
studied the instability of standing waves for the
nonlinear Schr\"odinger equation
\begin{equation}\label{nlsg}
iu_t+\D u+ |u|^{p-1}u=0
\end{equation}
where $u=u(t,x)\in\C$, $t\in\R$, $x\in\RN$ and $p>1$. A standing wave is a
solution of (\ref{nlsg}) of the form $e^{i\w t}\f(x)$ with $\f\in\hu$ and $\w>0$. Thus $\f$ is solution of
\begin{equation}\label{snlsg}
-\D \f+\w \f=|\f|^{p-1}\f,\quad \f\in\hu.
\end{equation}
We say that $\f \in \hu$ is a ground state solution of (\ref{snlsg}) if it
satisfies
$$
\tilde{S}(\f)= \inf\{ \tilde{S}(v);v\in\hu\setminus\{0\} \mbox{ is a solution of
(\ref{snlsg}) }\}
$$
where $\tilde{S}$ is defined for $v\in\hu$ by
$$
\displaystyle \tilde{S}(v) :=  \frac{1}{2}\nldd{\nabla
v}+\frac{\w}{2}\nldd{v}- \frac{1}{p+1}\intrn |v|^{p+1}dx.
$$
In \cite{bc1} it is shown that if $1+ \frac{4}{N} < p <1+
\frac{4}{N-2}$ when $N \geq 3$ and $1+ \frac{4}{N} < p<+\infty$ when
$N=1,2$, then any standing wave associated with a ground state solution $\f$ of (\ref{snlsg}) is unstable by blow up. More precisely, there  exists $(\f_n
)\subset \hu$ such that $\f_n \to \f$ in $\hu$ and the corresponding
maximal solution $u_n$ of (\ref{nlsg}) with $u_n (0) = \f_n$ blows
up in finite time.
\medskip

The result and perhaps more the methods introduced in \cite{bc1}
still have a deep influence on the field of instability
for nonlinear Schr\"odinger and related equations. In particular
the idea of defining appropriate invariant sets and how to use
them to establish the blow-up.  We should mention that in
\cite{bc1} more general nonlinearities were considered. The
paper \cite{bc1} is only a short note which contains the main ideas
but no proofs. For the special nonlinearity $|u|^{p-1}u$ these
proofs can be found in \cite{c}. For the general case it seems
that the extended version \cite{bc2} of \cite{bc1} has remained
unpublished so far.
\medskip

The aim of the present note is to present an alternative, shorter proof of the
result of \cite{bc1} for general nonlinearities. Also the instability of the standing waves is proved under
slightly weaker assumptions. Before stating our result we need to introduce
some notations. Let $g:\R\mapsto \R$ be an odd function extended to $\C$ by
setting $g(z)=g(|z|)z/|z|$ for $z\in\C\setminus\{ 0 \}$. Equation (\ref{nlsg})
now becomes
\begin{equation}\label{nlsgg}
iu_t+\D u+g(u)=0
\end{equation}
and correspondingly for the ground states we have
\begin{equation}\label{snlsgg}
-\D \f+\w \f=g(\f).
\end{equation}
For $z\in\C$ let $G(z):=\int_0^{|z|}g(s)ds$. We assume
\begin{itemize}
\item[$(A_0)$]
The function $g$ satisfies
\begin{itemize}
\item [(a)] $g\in\mathcal{C}(\R,\R)$.
\item [(b)] $\lim_{s \to 0}\frac{g(s)}{s}=0.$
\item [(c)]
\begin{itemize}
\item[]\hspace{-1cm} when $N\geq 3$, $\lim_{s\goesto+\infty}g(s)s^{-\frac{N+2}{N-2}}=0$; 
\item[]\hspace{-1cm} when $N=2$, for any $\a>0$, there exists $C_\a>0$ such that $|g(s)|\leq C_\a e^{\a s^2}$ for all $s>0$.
\end{itemize}
\end{itemize}
\item[$(A_1$)] The function $h(s):=(sg(s)-2G(s))s^{-(2+4/N)}$ is strictly increasing on $(0,+\infty)$ and
$\lim_{s\goesto 0}h(s)=0$.
\item[$(A_2)$] There exist $C>0$ and $\alpha \in [0, \frac{4}{N-2})$ if $N \geq 3$,  $\alpha \in [0, \infty)$ if $N=2$, such that $$|g(s)-g(t)|\leq
C(1+|s|^\alpha+|t|^\alpha)|t-s|$$ for all $s,t\in\R$. If $N=1$ we
assume that for every $M >0$, there exists $L(M) >0$ such that
$$ |g(s)-g(t)| \leq L(M) |s-t|$$
for all $s,t \in \R$ such that $|s| + |t| \leq M.$
\end{itemize}
Finally we define for $v \in \hu$ the functional
$$
\displaystyle S(v) :=  \frac{1}{2}\nldd{\nabla
v}+\frac{\w}{2}\nldd{v}- \intrn G(v) dx
$$
and set
$$
m := \inf\{ S(v);v\in\hu\setminus\{0\} \mbox{ is a solution of (\ref{snlsgg})
}\}.
$$

Our main result is

\begin{thm}\label{bc1thm}
Assume that $(A_0)-(A_2)$ hold and let $\f$ be a ground state
solution of (\ref{snlsgg}), i.e. a solution of (\ref{snlsgg}) such
that $S(\f) =m$. Then for every $\e>0$ there exists $\uo\in\hu$
such that $\nhu{\uo-\f}<\e$ and the solution $u$ of
(\ref{nlsgg}) with $u(0)=\uo$ satisfies
$$
\lim_{t\goesto \Tuo}\nld{\nabla u(t)}=+\infty\mbox{ with }
\Tuo<+\infty.
$$
\end{thm}

From \cite{bgk,bl} we know that assumption $(A_0)$ is almost
necessary to guarantee the existence of a solution for (\ref{snlsgg}). Assumption $(A_1)$ is a weaker version of
the assumption (H.1) in \cite{bc1}. An assumption of this type,
on the growth of $g$, is necessary since it is known from
\cite{cl} that when $g(u) = |u|^{p-1}u$ with $1 < p < 1 +
\frac{4}{N}$ the standing waves associated with the ground states are, on the contrary, orbitally
stable. Assumption $(A_2)$ is purely technical and is aimed at 
ensuring the local well-posedness of the Cauchy problem for (\ref{nlsgg}).
It replaces assumption (H.2) in \cite{bc1}. Indeed, in \cite{bc1} the authors were using the
results of Ginibre and Velo \cite{gv} for that purpose. Since \cite{bc1} has been
published, advances have been done in the study of the Cauchy
problem (see \cite{c,cw} and the references therein). In
particular, under our condition $(A_2)$, for all $\uo\in\hu$ there
exist $\Tuo>0$ and a unique solution
$u\in\mathcal{C}([0,\Tuo),\hu)\cap\mathcal{C}^1([0,\Tuo),\hmu)$
of (\ref{nlsgg}) such that $\lim_{t\goesto \Tuo}\nld{\nabla
 u(t)}=+\infty$ if $\Tuo<+\infty$. Furthermore, the following
conservation properties hold : for all $t\in[0,\Tuo)$ we have
\begin{eqnarray}
S(u(t))&=&S(\uo),\label{conservationenergyg}\\
\nld{u(t)}&=&\nld{\uo}.\label{conservationchargeg}
\end{eqnarray}
Finally, the function $f:t\mapsto\nldd{xu(t)}$ is $\mathcal{C}^2$
and we have the virial identity
\begin{equation}\label{virial}
\dtt f(t)  =  8Q(u(t)),
\end{equation}
where $Q$ is defined for $v\in\hu$ by
$$
Q(v):=\nldd{\nabla v}-\frac{N}{2}\intrn(g(|v|)|v|-2G(v))dx.
$$

The proofs of instability in \cite{bc1} and here share some
elements, in particular the introduction of sets invariant under
the flow. The main difference lies in the variational
characterization of the ground states which is used to define the
invariant sets and how to derive this characterization. \medskip

In \cite{bc1} it is shown that a ground state of (\ref{snlsgg}) can
be characterized as a minimizer of $S$ on the constraint
$$
M := \{v\in\hu\setminus\{0\},Q(v)=0\}.
$$
To show this characterization, $S$ is directly minimized on $M$.
Additional assumptions (see (H.1) in \cite{bc1}) are necessary at
this step to insure that the minimizing sequences are bounded.
Once the existence of a minimizer for $S$ on $M$ has been
established, one has to get rid of the Lagrange multiplier, namely
to prove that it is zero. There, a stronger version of $(A_0)$,
requiring in particular $g \in \mathcal{C}^1(\R,\R)$ and a control on
$g'(s)$ at infinity, is necessary along with tedious calculations.
\medskip

Having established that the ground states of (\ref{snlsgg}) 
minimize $S$ on $M$, Berestycki and Cazenave show that the
set
$$ K : = \{v \in \hu, S(v) < m \mbox{ and } Q(v) <0 \}$$
is invariant under the flow of (\ref{nlsgg}) and that one can choose in $K$ an initial data,
arbitrarily close to the ground state, for which the blow-up
occurs. \medskip

In our approach we characterize the ground states as  minimizers of
$S$ on
$$
\M:=\{ v\in\hu\setminus\{0\};Q(v)=0,I(v)\leq 0 \},
$$
where $I(v)$ is defined for $v \in \hu$ by
$$
I(v) := \nldd{\nabla v}+\w\nldd{v}-\intrn g(|v|)|v| dx
$$
and correspondingly our invariant set is
$$\
\{v \in
\hu, S(v) < m, Q(v) <0 \mbox{ and } I(v)<0 \}.$$

The dominant feature of our approach, which also explains why our
assumptions on $g$ are weaker than in \cite{bc1} is that we never
explicitly solve a minimization problem. At the heart of our
approach is an additional characterization of the ground states as
being at a mountain pass level for $S$. This characterization
was derived in \cite{jt1} for $N \geq 2$ and in \cite{jt2} for
$N=1$. We also strongly benefit from recent techniques developed
by several authors \cite{lffks,l1,l2,lww,ot,z} where minimization approches
using two constraints have been introduced.

\section{Proof of Theorem 1}

We first prove the existence of ground states and the fact that they correspond to minimizers of $S$ on the Nehari manifold.

\begin{lem}\label{lem:existencegroundstate}
Assume that $(A_0)$ and $(A_1)$ hold. Then (\ref{snlsgg}) admits a
ground state solution. Furthermore, the ground states solutions of (\ref{snlsgg})
are minimizers for
$$
\dw:=\inf\left\{ S(v);v\in\hu\setminus\{0\},I(v)=0 \right\}.
$$
\end{lem}

Before proving Lemma \ref{lem:existencegroundstate}, we prove a technical result.

\begin{lem}\label{lem:tech}
Assume that $(A_0)$ and $(A_1)$ hold. Then the nonlinearity $g$ satisfies
\begin{eqnarray}
\frac{g(s)}{s}\mbox{ is increasing for }s>0.&&\label{cond3}\\
\frac{g(s)}{s} \to + \infty \mbox{ as } s \to + \infty.&&\label{cond4}
\end{eqnarray}
\end{lem}

\begin{proof}[Proof of Lemma \ref{lem:tech}]
From the definition of $h(s)$ we have
\begin{equation}\label{cond1}
\frac{g(s)}{s}=s^{4/N}h(s)+\frac{2G(s)}{s^2}.
\end{equation}
Furthermore, for $s>0$
\begin{equation}\label{cond2}
\frac{\partial}{\partial
s}\left(\frac{G(s)}{s^2}\right)=\frac{s(sg(s)-2G(s))}{s^4}> 0
\end{equation}
where the last inequality follows from $(A_1)$. Thus, combining
(\ref{cond1}), (\ref{cond2}) and $(A_1)$ we get (\ref{cond3}) and (\ref{cond4}).
\end{proof}

\begin{proof}[Proof of Lemma \ref{lem:existencegroundstate}]
It follows from Lemma \ref{lem:tech} that 
\begin{itemize}
\item[(P)] There exists $s_0 >0$ such that 
\begin{itemize}
\item if $N \geq 2$, then $\frac{1}{2}\omega s_0^2 < G(s_0)$; 
\item if $N=1$, then $\frac{1}{2}\omega s^2 > G(s)$ for $s \in (0,s_0)$,
$\frac{1}{2}\omega s_0^2 = G(s_0)$ and $\omega s_0 < g(s_0).$
\end{itemize}
\end{itemize}
Now, from \cite[Théorème 1]{bgk} and \cite[Theorem 1]{bl} we know that the conditions $(A_0)$ and (P) are sufficient to insure the existence of a ground state.

If $v$ is a solution of (\ref{snlsgg}), then $ S'(v)v = I(v)=0$;
therefore, to prove the lemma it is enough to show that $\dw\geq
m$. From \cite{jt1,jt2} we know that under $(A_0)$ and (P) the
functional $S$ has a mountain pass geometry. More precisely, if we set 
$$
\G:=\{\k\in\mathcal{C}([0,1],\hu);\k(0)=0,S(\k(1))<0\},
$$
then $\G \neq \emptyset$ and 
$$
c:=\inf_{\k\in\G}\max_{t\in[0,1]} S(\k(t))>0.
$$
In addition it is shown\footnote{In fact, the results of \cite{jt1,jt2} are proved only for real valued functions; however, it is not hard to see that they can be extended to the complex case (see \cite[Lemma 14]{jl}).}
 in \cite{jt1,jt2}
that $$m=c.$$
Namely the \emph{mountain pass level} $c$ corresponds to the ground
state level $m$. Now it is well-known that (\ref{cond3}) ensure that
if $v\in\hu$ satisfies $I(v)=0$ then $t\mapsto S(tv)$ achieves its
unique maximum on $[0,+\infty)$ at $t=1$. Also (\ref{cond4}) shows
that $\lim_{t\goesto+\infty}S(tv)=-\infty$. From the definition of
$c$, it implies that $c\leq S(v)$ for all $v\in\hu$ such that $I(v)=0$.
Hence we have
$$
\dw\geq c,
$$
and combined with the fact that $m=c$ it ends the proof.
\end{proof}

Now we investigate the behavior of the functionals under some
rescaling
\begin{lem}\label{lemscaling}
Assume that $(A_0)$ and $(A_1)$ hold. For $\l>0$ and $v\in\hu$, we
define $\vl(\,\cdot\,):=\l^{\frac{N}{2}}v(\l\,\cdot\,)$. We suppose $Q(v)\leq 0$.
Then there exists $\lo\leq 1$ such that
\begin{enumerate}[{\rm(i)}]
\item $Q(\vlo)=0$,
\item $\lo=1$ if and only if $Q(v)=0$,
\item $\frac{\partial}{\partial \l}S(\vl)>0$ for $\l\in(0,\lo)$ and $\frac{\partial}{\partial \l}S(\vl)<0$ for $\l\in(\lo,+\infty)$,
\item $\l\mapsto S(\vl)$ is concave on $(\lo,+\infty)$,
\item $\frac{\partial}{\partial \l}S(\vl)=\frac{1}{\l}Q(\vl)$.
\end{enumerate}
\end{lem}

\begin{proof}[Proof of Lemma \ref{lemscaling}]
Easy computations lead to
\begin{eqnarray*}
\frac{\partial}{\partial \l}S(\vl)& = & \frac{1}{\l}Q(\vl)\\
&=& \l\left( \nldd{\nabla v}-\frac{N}{2}\intrn
\lambda^{-(N+2)}\left(\lnd g(\lnd |v|)|v|-2G(\lnd v)\right)dx
\right),
\end{eqnarray*}
and recalling from $(A_1)$ that the function
$h(s):=(sg(s)-2G(s))s^{-(2+4/N)}$ is strictly increasing on $[0,+\infty)$,
(i),(ii),(iii) and (v) follow easily. To see (iv), we remark that
since
$$
\left( \nldd{\nabla v}-\frac{N}{2}\intrn \lambda^{-(N+2)}
\left(\lnd g(\lnd |v|)|v|-2G(\lnd v)\right)dx \right)<0
$$
on $(\lo,+\infty)$, we infer from $(A_1)$ that $\frac{\partial}{\partial \l}S(\vl)$ is strictly
decreasing on $(\lo,+\infty)$, which implies (iv).
\end{proof}

\begin{proof}[Proof of Theorem \ref{bc1thm}]
We recall that
$$
\M=\{ v\in\hu\setminus\{0\};Q(v)=0,I(v)\leq 0 \},
$$
and define
$$
\dM:=\inf\{ S(v);v\in \M \}.
$$
We proceed in three steps.\\
{\sc Step 1.} Let us prove $\dw=\dM$. Since the ground states $\f$ satisfy $Q(\f)=I(\f)=0$, we have $\f\in\M$. Combined with $S(\f)=\dw$, this implies $\dM\leq\dw$. Conversely, let $v\in\M$. If $I(v)=0$, then trivially $S(v)\geq\dw$, thus we suppose $I(v)<0$.
We use the rescaling defined in Lemma \ref{lemscaling} : for $\l>0$ we have
$$
I(\vl)=\l^2\nldd{\nabla v}+\w\nldd{v}-\intrn\l^{-N/2}g(\l^{N/2} |v|)|v|dx.
$$
It follows from $(A_0)$-(b) that $\lim_{\l\goesto
0}I(\vl)=\w\nldd{v}$ and thus by continuity there exists $\lu<1$
such that $I(\vlu)=0$. Thus $S(\vlu)\geq \dw$. Now, from $Q(v)=0$
and (iii) in Lemma \ref{lemscaling} we have
$$S(v)\geq S(\vlu)\geq \dw,$$ hence $\dM=\dw$.

\medskip

{\sc Step 2.} For $\l>0$, we set $\ul:=\fl$. For $\l>1$ close to
$1$, we have
\begin{eqnarray}
&&S(\ul)<S(\f)\mbox{ and }Q(\ul)<0,\label{36}\\
&&I(\ul)<0\label{37}.
\end{eqnarray}
Indeed, (\ref{36}) follows from (iii) and (v) in Lemma \ref{lemscaling}. For (\ref{37}), we write
\begin{eqnarray*}
I(\ul)&=&2S(\ul)+\frac{2}{N}Q(\ul)-\frac{2}{N}\nldd{\nabla \ul}\\
&\leq&2S(\f)+\frac{2}{N}Q(\f)-I(\f)-\frac{2\l^2}{N}\nldd{\nabla \f}\\
&\leq&\frac{2(1-\l^2)}{N}\nldd{\nabla \f}<0.
\end{eqnarray*}

Let $u(t)$ be the solution of (\ref{nlsgg}) with $u(0)=\ul$. We
claim that the properties described in (\ref{36}), (\ref{37}) are
invariant under the flow of (\ref{nlsgg}). Indeed, since from
(\ref{conservationenergyg}) we have for all $t>0$
\begin{equation}\label{prop2}
S(u(t))=S(\ul)<S(\f),
\end{equation}
we infer that  $I(u(t))\neq 0$ for any $t\geq 0$, and by
continuity  we have $I(u(t))< 0$ for all $t\geq 0$. It follows
that $Q(u(t))\neq 0$ for any $t\geq0$ (if not $u(t)\in \M$ and
thus $S(u(t)) \geq S(\f)$ which contradicts (\ref{prop2})), and by
continuity we have $Q(u(t))<0$ for all $t\geq 0$. Thus for all
$t>0$ we have
$$
S(u(t))<S(\f),I(u(t))<0\mbox{ and }Q(u(t))<0.
$$

\medskip

{\sc Step 3.} We fix $t>0$ and define $v:=u(t)$. For $\b>0$, let $\vb(x):=\b^{\frac{N}{2}}v(\b x)$. From {\sc Step 2} we have $Q(v)<0$, thus from Lemma \ref{lemscaling} there exists $\bo<1$ such that $Q(\vbo)=0$. If $I(\vbo)\leq 0$, we keep $\bo$, otherwise we replace it by $\bot\in(\bo,1)$ such that $I(\vbot)=0$. Thus in any case we have
\begin{equation}\label{40}
S(\vbo)\geq \dw
\end{equation}
and $Q(\vbo)\leq 0$. Now from (iv) in Lemma \ref{lemscaling}, we have
$$
S(v)-S(\vbo)\geq (1-\bo)\frac{\partial}{\partial \b}S(\vb)_{|\b=1}.
$$
Thus, from (v) in Lemma \ref{lemscaling}, $Q(v)<0$ and $\bo<1$, we get
$$
S(v)-S(\vbo)\geq Q(v).
$$
Combined with (\ref{40}), this gives
\begin{equation}\label{41}
Q(v)\leq S(v)-\dw:=-\d<0
\end{equation}
where $\d$ is independent of $t$ since $S$ is a conserved quantity.

To conclude, it suffices to observe that thanks to (\ref{virial}) and (\ref{41}) we have
\begin{equation}\label{42}
\nldd{xu(t)}\leq -\d t^2+Ct+\nldd{x\ul},
\end{equation}
and since the right hand side of (\ref{42}) becomes negative when $t$ grows up, we easily deduce that $T_{\ul}<+\infty$ and
$
\lim_{t\goesto T_{\ul}}\nld{\nabla u(t)}=+\infty.
$
\end{proof}

\textbf{Acknowledgements.} The author is grateful to Louis Jeanjean for helpful advice during the writing of this paper. He also wishes to thank Thierry Cazenave for sharing with him a digitalized version of the unpublished paper \cite{bc2}.


\begin{thebibliography}{99}

\footnotesize



\bibitem{bc1}
{\sc H. Berestycki and T. Cazenave,} {\em Instabilit\'e des
\'etats stationnaires dans les \'equations de Schr\"odinger et de
Klein-Gordon non lin\'eaires}, {C. R. Acad. Sci. Paris} {\bf 293},
(1981), 489--492.

\bibitem{bc2}
{\sc H. Berestycki and T. Cazenave,} {\em Instabilit\'e des
\'etats stationnaires dans les \'equations de Schr\"odinger et de
Klein-Gordon non lin\'eaires}, {Publications du Laboratoire
d'Analyse Num\'{e}rique,} Universit\'e de Paris VI.

\bibitem{bgk}
{\sc H. Berestycki, T. Gallouet and O. Kavian,}
{\em \'Equations de champs scalaires euclidiens non lin\'eaires dans le plan,}
C. R. Acad. Sci. Paris \textbf{297}, (1983), 307-310.

\bibitem{bl}
{\sc H. Berestycki and P.L. Lions,} {\em Nonlinear \ scalar \  field\ equations
I,} {Arch. Ration. Mech. Anal.,} {\bf 82}, (1983), 313--346.

\bibitem{c}
{\sc T. Cazenave,} {\em Semilinear Schr{\"o}dinger equations},
Courant Lecture Notes in Mathematics, {\bf 10}, (2003).

\bibitem{cl}
{\sc T. Cazenave and P.L. Lions,} {\em Orbital stability of standing waves for
some nonlinear Schr\"odinger equations}, {Comm. Math. Phys.}, {\bf 85}, 4,
(1982), 549--561.

\bibitem{cw}
{\sc T. Cazenave and F.B. Weissler}, {\em The Cauchy problem for
the nonlinear Schr\"odinger equation in $H^1$}, Manuscripta Math.
{\bf 61}, (1988), 477-494.


\bibitem{gv}
{\sc J. Ginibre and G. Velo,} {\em On a class of nonlinear
{S}chr\"odinger equations. {I}. {T}he {C}auchy problem, general
case}, {J. Func. Anal.} {\bf 32}, 1, (1979), 1--32.

\bibitem{jl}
{\sc L. Jeanjean and S. Le Coz,}
{\em Instability for standing waves of nonlinear Klein-Gordon equations via mountain-pass arguments},
\textsf{preprint}.

\bibitem{jt1}
{\sc L. Jeanjean and K. Tanaka,}
{\em A note on a mountain pass characterization of least energy solutions},
Adv. Nonlinear Stud. \textbf{3}, 4, (2003), 445--455.

\bibitem{jt2}
{\sc L. Jeanjean and K. Tanaka,}
{\em A remark on least energy solutions in $R\sp N$},
Proc. Amer. Math. Soc. \textbf{131}, 8, (2003), 2399--2408.

\bibitem{lffks}
{\sc S. Le Coz, R. Fukuizumi, G. Fibich, B. Ksherim and Y. Sivan,}
{\em Instability of bound states of a nonlinear Schrödinger equation with a Dirac potential}, preprint.

\bibitem{l1}
{\sc Y. Liu,}
{\em Blow up and instability of solitary-wave solutions to a generalized Kadomtsev-Petviashvili equation},
Trans. Amer. Math. Soc. \textbf{353}, 1, (2001), 191--208.

\bibitem{l2}
{\sc Y. Liu,}
{\em Strong instability of solitary-wave solutions to a Kadomtsev-Petviashvili equation in three dimensions},
J. Differential Equations \textbf{180}, 1, (2002), 153--170.

\bibitem{lww}
{\sc Y. Liu, X.-P. Wang and K. Wang,}
{\em Instability of standing waves of the Schr\"{o}dinger equation with inhomogeneous nonlinearity},
Trans. Amer. Math. Soc. \textbf{358}, (2006), 2105--2122.


\bibitem{ot}
{\sc M. Ohta and G. Todorova,}
{\em Strong instability of standing waves for nonlinear {K}lein-{G}ordon equations},
Discrete Contin. Dyn. Syst.
\textbf{12}, 2, (2005), 315--322.


\bibitem{z}
{\sc J. Zhang,} {\em Sharp threshold for blowup and global existence in nonlinear
Schr\"{o}dinger equations under a harmonic potential,} {Comm. Partial Differential Equations}, {\bf 30}, 10-12, (2005), 1429--1443.



\end{thebibliography}
\end{document}